\documentclass[10 pt,reqno]{amsart}

\usepackage[english]{babel}

\usepackage{verbatim, latexsym, amssymb, amsmath}

\usepackage{amsfonts}

\def \R {\mathbb{R}}

\def \Si {\Sigma}

\newcommand{\cone}{{\times\!\!\!\!\times}}

\newtheorem{theorem}{Theorem}
\newtheorem{lemma}[theorem]{Lemma}
\newtheorem{proposition}[theorem]{Proposition}

\theoremstyle{definition}
\newtheorem{definition}{Definition}

\newtheorem{remark}{Remark}
\numberwithin{equation}{section}

\begin{document}

\title[Capacity  and Rigidity of Cornered Manifolds]{ Capacity inequalities and rigidity of\\
 cornered/conical manifolds
}

\thanks{This work was completed with the support of Cnpq/Brazil.}


\author{Tiarlos Cruz}
\address{%
Universidade Federal de Alagoas \\
 Instituto de Matem\'atica \\
 Macei\'o, AL -57072-970\\
 Brazil}
\email{cicero.cruz@im.ufal.br}


\begin{abstract}
We prove capacity inequalities involving the total mean curvature of hypersurfaces with boundary in convex cones and the mass of asymptotically flat manifolds with non-compact boundary. We then give the analogous of P\"olia-Szeg\"o, Alexandrov-Fenchel  and Penrose type inequalities in this setting. Among the techniques used in this paper are the inverse mean curvature flow for hypersurfaces  with boundary.
 \end{abstract}

\maketitle

\section{Introduction and statement of the results} \label{Sct intro}

In this paper we aim to provide some new integral inequalities in terms of well known geometric quantities. Some of them  are closely related to the isoperimetric problem  in the theory of convex cones.

 On the geometric side, we are interested in estimating the capacity of hypersurfaces in terms of the total mean curvature and the mass of asymptotically flat manifolds with non-compact boundary.  
	We point out that there are still few examples where the exact value of the capacity of a set is known, thus estimates by using geometric terms are of great interest. The most known sharp capacity estimates were obtained by Szeg\"o \cite{Sz, Sz2}, also including  rigidity statements.  From a physical point of view, 
	the capacity of a compact set $\Omega$ in a Riemannian manifold $M$ represents the electric charge flowing into $M\backslash \Omega$ through the boundary $\partial \Omega$ so that the electric potential of the field created by this charge is bounded by $1$, see e.g. \cite[$\mathsection$ 2.1]{PS}.

In part of this work, we focus on convex cones of $\mathbb{R}^{n}$ and our motivation is due to the existence of a lot of interest in quantitative estimates for isoperimetric inequalities and  existence of isoperimetric regions in such setting. Relevant contributions in this direction were obtained by several mathematicians including P. L. Lions, F. Pacella, J. Choe, F. Morgan, M. Ritor\'e, C. Rosales, A. Figali and E. Indrei (see, for instance,   
\cite{LP3,C,MR,RS,RV, FI} and references therein, as well as relevant discussion
in \cite{PT,LP}). 

Another important case treated here involves the mass of asymptotically flat manifolds with non-compact boundary, which is an invariant quantity of the asymptotic geometry. We make use of the mass to give upper bounds for the capacity. In fact, this is a natural question which was first studied by Bray in \cite{B} and later by Bray and Miao \cite{BM}, using the positive mass theorem and the monotonicity of the Hawking mass along the inverse mean curvature flow (IMCF), respectively.  Recently, a positive mass theorem on manifolds admitting non-compact boundary was settled by Almaraz, Barbosa and de Lima \cite{ABdL}, by adapting the classical  method established  by Schoen and Yau \cite{SY1,SY2} and Witten \cite{W}.

Building on the ideas of Freire and Schwartz  in their proof of mass-capacity inequalities  \cite{FS}, we obtain analogous results  by considering  now hypersurfaces with boundary evolving under inverse
mean curvature flow, an approach introduced by Marquadt \cite{M2, M3},  who constructed solutions  by rewriting the flow as an equation for the level set of a function whose advantage is to allow ``jumps'' in a natural way. For different approaches  of this geometric flow, we refer the reader to \cite{LS,LS2,M3}.

In what follows, let us consider the Euclidean cone  $C=\vec 0\cone U,$ where $U$ is a domain of the sphere $\mathbb S^{n-1},$ so the boundary $\partial C$ is the union of geodesic segments from vertex $\vec 0$ to the points of $\partial U.$ We let $\alpha_{n-1}$ denote  the volume of the unitary sector $C(\alpha,1):=C\cap\{|x|\leq1\}$ with solid angle $\alpha.$ Observe that if $U$ is an open half-sphere the cone coincides with an open half-space. 

Let $\Omega\subset C$ be a smooth domain. It is important to distinguish the boundary parts of $\Omega$ on $\partial C$ and in the interior of $C$ by writing 
\begin{equation}\label{key}
	\partial_{C}\Omega:=\overline{\partial \Omega\backslash \partial C}\;\;\;\textrm{and}\;\;\;\partial_{\partial C}\Omega:=\partial \Omega\backslash \partial_{C}\Omega.
\end{equation}
	
The  Dirichlet energy of a map $\phi:C\backslash \Omega\to\R$ is defined as  
$$
E(\phi):=\frac{1}{(n-2)\alpha_{n-1}}\int_{C\backslash \Omega}|D \phi|^2dv,
$$
where $dv$ is the volume element of $C\backslash \Omega$. 

We are now ready to give the following definition of capacity.	
	\begin{definition}
		
		Let $C$ be a cone centered at the origin.
		Assume that $\Omega\subset C$ is  a smooth bounded domain containing the vertex of the cone  such that $\Sigma=\partial_C \Omega$  meets $\partial C$ orthogonally. The {\it Capacity} of $\Sigma$  is given by 
		\begin{equation}\label{defcap}
		\mbox{Cap}(\Sigma)=\inf E(\phi),
		\end{equation}
		where the infimum is taken over all smooth  functions $\phi:C\backslash \Omega\rightarrow \R$ with  $\phi_{|_{\Sigma}}=0$ and approach to one at infinity. 
		
			\end{definition}

We state a rigidity result for free boundary outer-minimizing  mean convex 
domains (not necessarily connected) in convex cones (a P\'olya-Szeg\"o type inequality as in \cite[$\mathsection$ 2(15)]{Sz2}). Convexity of the cone here means  that the second fundamental form $\Pi$ of  $\partial C\backslash\vec 0$ with respect to the outward unit normal $\mu$ is non-negative. 

\begin{theorem}\label{main3}	
	Let $C\subset \mathbb{R}^{n}$ be a smooth  convex cone centered at the origin. Assume that  $\Omega$ is a smooth bounded domain containing the vertex of $C$ such that  $\Sigma=\partial_C\Omega$ meets $\partial C$ orthogonally and is strictly mean convex, i.e., $H>0$. If one of the following hypotheses holds.
	\begin{itemize}
		\item[I)]  $n\leq7;$ or
		\item[II)]  $n>7$, $\Sigma$ is free boundary outer-minimizing in $C\backslash\Omega$;
	\end{itemize}
	Then 
	\begin{equation*}
		\mbox{Cap}(\Sigma)\le  \frac 1{(n-1)\alpha_{n-1}}\int_{\Sigma}Hd\sigma.
	\end{equation*}
	Equality holds if and only $\Omega$ 
	is  the intersection of a ball centered
	at the vertex of cone with $C$. 
	%In particular, $\partial\Omega$ is a spherical cap.
\end{theorem}

\begin{remark}
	The above result can be proved under the same hypothesis but with outer-minimizing replaced by star-shaped with
respect to the center of $C$. In which case one can use the same approaching as in Theorem \ref{main4}.  
\end{remark}

 In the theory of convex bodies, Alexandrov-Fenchel inequalities to star-shaped Euclidean domains  with mean convex boundary have recently generated a fair amount of interest, see e.g. \cite{GL} (see also \cite{FS}, for the case of domains with outer-minimizing boundary). In this work, we are also able to prove  estimations of the total mean curvature of a star-shaped hypersurfaces  in convex cones. The precise statement is the following.

\begin{theorem}\label{main4}	
	Let $C\subset \mathbb{R}^{n}$ be a smooth convex cone centered at the origin. Assume that  $\Omega$ is a smooth bounded domain containing the vertex of $C$ such that  $\Sigma=\partial_C\Omega$ meets $\partial C$ orthogonally and is strictly mean convex, i.e., then $H>0$.  If $\Sigma$ is star-shaped  in $C\backslash\Omega,$
	 we have
	\begin{equation}\label{al-fen}
	\frac1{2(n-1)\alpha_{n-1}}\int_{\Sigma}Hd\sigma\ge \frac{1}{2}\left(\frac{\mbox{Area}(\Si)}{\alpha_{n-1}}\right)^{\frac{n-2}{n-1}},
	\end{equation}
	with equality achieved if and only if  $\Omega$ 
	is  the intersection of a ball centered
	at the vertex of cone with $C$.
	% In particular, $\partial\Omega$ is a spherical cap.
\end{theorem}

\begin{remark}
	It would be interesting to know if Theorem \ref{main4} can be generalized to $k$-convex starshaped domains in the sense of \cite{GL}.
\end{remark}

Next we give a version of the Poincar\'e--Faber--Szeg\"o  inequality  relating the capacity of a region to its volume (the Euclidean corresponding result can be found in \cite[$\mathsection$ 2]{Sz}).

\begin{theorem} \label{PFS} 
Let $C\subset \mathbb{R}^{n}$ be a smooth convex cone centered at the origin. Assume that  $\Omega$ is a smooth bounded domain containing the vertex of $C$ such that  $\Sigma=\partial_C\Omega$ meets $\partial C$ orthogonally. If $C\backslash\Omega$ is connected, then
	\begin{equation}
	\label{ineqcapl}
	\mbox{Cap}(\Si)\geq \left(\frac{\mbox{Vol}(\Omega)}{\mbox{Vol}(C\cap\{|x|\leq1\})}\right)^{\frac{n-2}{n}},
	\end{equation}
	with equality achieved if and only if  $\Omega$ is  the intersection of a ball centered at the vertex of cone with $C$. 
	%In particular, $\partial\Omega$ is a spherical cap.
\end{theorem}

We say that the Riemannian manifold $M$ is  {\em cornered}  ({\em or curve-faced polyhedral}) {\em manifold} of depth 
$d=2$ if the boundary $\partial M$ of $M$ is decomposed into a union of faces $\Si_1\cup\Si_2,$ such that $\Si_1$ is transversal to $\Si_2$ and the boundary $\partial\Si_i$ of $\Si_i$ is equal to $\Si_1\cap\Si_2$ for $i=1,2$, see \cite{Gr} 
for further discussion of cornered manifolds. Finally,
we should mention that the Riemannian half Schwarzschild space (see Section \ref{Model} for definition) is a non-trivial example of cornered manifold of depth 2.

In order to state our next results more precisely, let us introduce some terminology. A cornered manifold $(M^n,g)$, $n\geq 3$, of depth $2$ is said to be {\it conformally flat} if it is isometric to $( \mathbb{R}^n_+\backslash\Omega, u^ {\frac{4}{n-2}}\delta),$ where  $\Omega\subset  \mathbb{R}^n_+$ is a smooth bounded set such that $\partial\Omega$ intersect $\partial \mathbb{R}^n_+$ orthogonally. Furthermore, assume that $u$ normalized so that $u\to 1$ at $\infty.$ Let $\Si\cup S$ denote the boundary of $M.$ We consider the space $\mathcal{M}$ of all conformally flat metrics $g$ on $M$ such that:

	\begin{itemize}
		\item[i)]   The scalar curvature of $g=u^ {\frac{4}{n-2}}\delta$, denoted by $R_g$, is non-negative. 
			\item[ii)]  $\Sigma=\partial_{\mathbb{R}^n_+} \Omega$  and $S=\partial_{\partial\mathbb{R}^n_+}(\mathbb{R}^n_+\backslash\Omega)$ are mean convex and minimal with respect to the euclidean metric, respectively. 
		\item [iii)] $\Sigma$ is minimal  and $S$  is mean convex with respect to $g$. 
	
	\end{itemize}

The following theorem presents a  mass-capacity inequality and  a volumetric Penrose type  inequality  for conformally flat manifolds. Observe that the result holds in all dimensions.

\begin{theorem} \label{m}
Let $\mathfrak m(g)$ be the mass of $(M,g).$ Assume that  $(M,g)$ is an asymptotically flat manifold with non-compact boundary such that $g\in \mathcal{M}$. 
If  $u_{|_\Sigma}\geq 2$, we have:

\begin{equation}\label{ineq1}
\mathfrak m(g)\ge \mbox{Cap}(\Sigma,g),
\end{equation}
\begin{equation}\label{ineq2}
\mathfrak m(g)\geq 2\left(\frac{\mbox{Vol}(\Omega)}{\mbox{Vol}(\mathbb R_+\cap\{|x|\leq 1\}}\right)^{\frac{n-2}n}.
\end{equation}
Equality holds in (\ref{ineq1}) or (\ref{ineq2}) if and only if $g$ is the Riemannian half Schwarzschild metric. 
	
\end{theorem}

\begin{remark}\label{Rem}
	The  spacetime Penrose inequality is a long-standing conjecture  that has only been proved in a few cases. For instance, the Riemannian  version in dimension three was proved by Huisken and Ilmanen \cite{HI} and by Bray \cite{B}. It gives a relationship between the ADM mass of an end of the manifold  and the area of each outermost minimal sphere bounding the end.  
\end{remark}

\begin{remark}
	Let $\underline{u}$ be the inner radius of a hypersurface $\Si$ i.e. the infimum of the function $u$ on $\Si.$  We point out that, in fact, we can obtain the following inequalities:
	$$\mathfrak m(g)\ge \frac{2\underline u }{2+\underline u }\mbox{Cap}(\Sigma,g) \mbox{ \;and\; } \mathfrak m(g)\geq \underline u\left(\frac{\mbox{Vol}(\Omega)}{\mbox{Vol}(\mathbb R_+\cap\{|x|\leq 1\}}\right)^{\frac{n-2}n}.$$
\end{remark}

\begin{remark}
	More details, examples and importance of cornered domains can be found in the work of Gromov \cite{G}, whose approach allow us to see closed manifolds as  cornered manifolds of depth $0$ and those with nonempty  boundary as cornered manifolds of depth 1. 
\end{remark}

The paper is organized as follows. In Section \ref{simcfb}, 
we give an overview about the IMCF for hypersurfaces with boundary. Then we provide an interplay between  this geometric flow and the total mean curvature whose relationship plays a key role in this work.  Afterwards, in Section \ref{icmtc}, we relate the capacity and  total mean curvature of hypersurfaces with boundary which meet a cone perpendicular.  As a consequence we prove Theorem \ref{main3},  \ref{main4} and \ref{PFS}. The last section  is devoted to  establish some definitions and gather results in order to prove Theorem \ref{m}.

%%%%%%%%%%%%%%%%%%%%%%%%%%%%%%%%%%%%%%%%%%%%%
%%%%%%%%%%%%%%%%%%%%%%%%%%%%%%%%%%%%%%%%%%%%%

\section{ Total mean curvature and the IMCF for hypersurfaces with boundary}\label{simcfb}

%%%%%%%%%%%%%%%%%%%%%%%%%%%%%%%%%%%%%%%%%%%%%
%%%%%%%%%%%%%%%%%%%%%%%%%%%%%%%%%%%%%%%%%%%%%

\noindent 
In this section, we give a brief discussion of the inverse mean curvature flow for hypersurfaces that possesses boundary. Part of the proofs herein relies on modifications of the argument in 
\cite{FS} using the approach developed by Marquadt in \cite{M, M2, M3}.

Let $\Si$ be a compact, smooth, orientable, manifold with compact, smooth boundary $\partial \Sigma.$ Suppose that $X_0: \Sigma\to M$ is a $C^{2,\alpha}$-immersion such that $\Sigma_0:= X_0(\Sigma)$ has strictly positive mean curvature and is perpendicular to a fixed supporting $C^{2,\alpha}$-hypersurface $S$ without boundary in $(M,g),$ satisfying 
$$
X_0(\partial\Sigma)=X_0(\Sigma)\cap S\quad\textrm{and}\quad g(N_0,\mu\circ X_0)=0\;\;\textrm{on}\;\;\Sigma ,
$$ 
where $N_0$ and $\mu$ are the unit normal vector fields on $\Sigma$ and $S$, respectively.
Let $X:\Sigma\times[0,T]\to M$  be a solution of  IMCF for hypersurfaces with boundary

\begin{equation}\label{icmf1}
\begin{cases}
\frac{\partial X}{\partial t}=\frac{1}{H}N\,&\text{in}\: \;\Si\times(0,T)\\
X(\partial\Sigma,t)=\Sigma_t\cap S,\quad g(N,\mu\circ X_0)& \text{on}\;\;\partial\Si\times(0,T)\\
X(\cdot,0)=X_0,  &\text{on}\;\;\Si,
\end{cases}
\end{equation}
where  $H$, assumed to be positive, is the mean curvature of $\Si$ in $M$  with respect to $N$ and $\Si_t=X(\Si,t).$ Since $\Si$ is orthogonal to $S$, the outward unit co-normal $\nu$ of $\partial \Si$ coincides with $\mu$ along $\partial \Si$.

In particular when $M$ is a convex cone and $\Si$ is star-shaped with respect to the center of the cone, we have the following analogous statement to the one of Gerhardt \cite{G} for closed hypersurfaces.

\begin{theorem}[Marquadt \cite{M2}]\label{marqcone}
	Let $C\subset \mathbb{R}^{n}$ be a smooth convex cone centered at the origin. Let $X_0:\Si\to\mathbb{R}^{n}$ such that $\Si_0:=X_0(\Si)$ is a compact $C^{2,\alpha}$-hypersurface which is star-shaped with respect to the center of the cone and has strictly positive mean curvature. Furthermore, assume that $\Si_0$ meets  $\partial C$  orthogonally. Then there exists a unique embedding
	$$X\in C^{2+\alpha,1+\frac{\alpha}{2}}(\Si\times[0,+\infty);\mathbb{R}^{n})\cap C^{\infty}(\Si\times(0,+\infty);\mathbb{R}^{n})$$ 
	with $X(\partial\Si,t)\subset \partial C$ for all $t\geq 0$ satisfying (\ref{icmf1}). Furthermore, the rescaled embedding  converges smoothly to an embedding $X_{\infty}$,  mapping $\Si$ into a piece of a round sphere of radius $r_{\infty} =\Big(\frac{\mbox{Area}(\Si_0)}{\mbox{Area}(\Si)}\Big)^{1/(n-1)}$.
\end{theorem}

Now, we return to the general context. 
Let $\varphi:M\to\mathbb{R}^n$ be a function such that $\Si_t=\partial_M\Omega_t$ where $\Omega_t=\{\varphi<t\}.$ As long as the  mean curvature of $\Sigma_t$ is strictly positive, the parabolic formulation
(\ref{icmf1}) is equivalent to  
\begin{equation}\label{imcf2}
\begin{cases}
\textrm{div}\Big(\frac{D\varphi}{|D\varphi|}\Big)=|D\varphi|&\text{in}\: \;M_0:=M\backslash\overline{\Omega_0}\\
D_\mu \varphi=0& \text{on}\;\;\partial M_0:=\partial_{S}M_0\\
\varphi=0& \text{on}\;\;\partial_{M}\Omega_0.
\end{cases}
\end{equation}

The hypersurface flows in the outward normal direction with speed $N=\frac{1}{H}$ and it is easy to see that there is a problem if either $H=0$ or  $H$ changes the sign on $\Si_t.$\footnote{For instance, for non-star-shaped initial hypersurfaces, singularities may occur in finite time} To overcome theses problems, Marquadt \cite{M, M3} developed the notion of
weak solutions for (\ref{imcf2}), proving the existence and uniqueness of such solutions guided by the ideas of Huisken and Ilmanen \cite{HI}.

 Consider a foliation $\{\Si_t\} $ defined by the level sets of the function given by weak solution of IMCF in $M\backslash \overline \Omega_0$.  By Lemma 5.1 and 5.3 of \cite{M3}, each $\Si_t:=\partial_M \{\varphi<t\}$, $\varphi\in C^{0,1}_{loc}(M)$, is a $C^{1,\frac{1}{2}}$-hypersurface  up to a set of dimension less than or equal to $n-8$ which possesses a weak mean curvature in $L^{\infty}$ given by
$$
\overrightarrow H(x)=|D\varphi(x)|N(x),\quad\textrm{where}\quad N(x):=\frac{D\varphi(x)}{|D\varphi(x)|}
$$
for almost every $t>0,$ $x\in \Sigma_t$. Although the result originally stated in Lemma 5.1 of \cite{M3} does not includes any information about the regularity of the boundary, it can be extended up to the boundary applying Proposition 3.2 of \cite{V} and the Gr\"uter-Jost's free boundary regularity \cite{GJ}. 	Hence for those above values of $t$, $\Si_t$ is orthogonal to $S$ in the classical sense and any neighborhood of points $x\in S\cap \partial_M^{*} \Omega_t $\footnote{$\partial^*$ represents the reduced boundary in the sense of the set of locally finite perimeter in $\mathbb R^n$.}

\begin{definition}\label{fouter}
	Let $\Sigma\subset M$ be a hypersurface that meets $S$ orthogonally. $\Si$ is called {\it free boundary outer-minimizing} if any other hypersurface $\hat\Si$ that is transversal to $S$  and encloses $\Si$ has $$\mbox{Area}(\Si)\leq \mbox{Area}(\hat\Si).$$ We also say that $\Si$ is free boundary strictly outer-minimizing if every hypersurface which encloses it and is transversal to $S$ has strictly greater area.
	\end{definition}

Note that a free boundary outer-minimizing has $H\geq0$, since otherwise there would exist an  outward variation which would decrease its area. The following lemma gives the connection between  outer-minimizing property and parabolic problem with Neumann boundary condition (\ref{imcf2}).

\begin{lemma}
	If $\varphi$ is a solution of the equation (\ref{imcf2}), then $\Si_t$ is free boundary outer-minimizing  for all $t>0$. 
\end{lemma}
\begin{proof}
	Assume that $\Si$ is any hypersurface enclosing $\Si_t$ and let $U$ be the region between $\Si$ and $\Si_t$. Integrating by parts gives
	$$
	\int_{U}\mbox{div}\Big(\frac{D\varphi}{|D\varphi|}\Big)dv=\int_{\Si}\Big\langle N,\frac{D\varphi}{|D\varphi|}\Big\rangle d\sigma-\mbox{Area}(\Si_t)+\int_{\partial _{S}U}\frac{1}{|D\varphi|}D_{\mu}\varphi dl.
	$$
	Since the left hand side is  equal to $\int_{U}|D\varphi|dv\geq0$ and $D_{\mu}\varphi=0$ on $\partial _{S}U$, we get  $$\mbox{Area}(\Si_t)\leq \mbox{Area}(\Si)$$
	for all $t>0$.
\end{proof}

\begin{remark}
Note that $\varphi\equiv t$ on $U,$ when $\mbox{Area}(\Si)=\mbox{Area}(\Si_t)$.
\end{remark}

\begin{remark}
In the two-dimensional case, the existence of free boundary outer minimizing sets follows from the Plateau's problem with partially free boundary, see for example \cite{Cr}.

\end{remark}

In the next step we calculate the evolution of the total mean curvature under the flow (\ref{icmf1}). Before
proceeding, we need the following technical lemma.

\begin{lemma}\label{tecn}
	Let $C\subset \mathbb{R}^{n}$ be a smooth cone. Let $\Omega\subset C$ be a smooth bounded domain so that $\Si=\partial_C\Omega$ meets $\partial C$ orthogonally.  Consider a foliation $\{\Si_t\}_{t\geq0}$  given by  weak solution of IMCF  in $C\backslash\Omega$, where $\Si_t=\partial_C\{\varphi\leq t\}.$
	
	If $\Phi\in C^{0,1}_ {c}((0,t),\mathbb{R}_+^n)$, we have
	$$-\int_{\Omega_t}\langle D \phi, N\rangle Hdv 
	\leq \frac{n-2}{n-1}\int_{\Omega_t}\phi H^2dv -\int_{\partial_{\partial C}\Omega_t}\frac{\phi}{H}D_\mu Hdl$$
	where $\Omega_t=\{\varphi\leq t\}$ and $\phi=\Phi\circ \varphi:\Omega_t\rightarrow \mathbb{R}$.
\end{lemma}

\begin{proof}
	
	First, one can argue as Lemma A.1 in \cite{FS} (even supposing  $\varphi$ as a function of class $C^3$) to obtain the following integral expression:
	$$
	\int_{\Sigma_t}\frac 1{|D \varphi|}\langle N, D H\rangle d\sigma
	=\int_{\Sigma_t} \Delta_{\Sigma_t}(-|D \varphi|^{-1})d\sigma-\int_{\Sigma_t}\frac 1{|D \varphi|}|A|^2d\sigma,
	$$
	where $|A|$ is the square sum of the principal curvature of $\Sigma_t$. In the following the co-area formula, the divergence theorem and the fact that $H = |D\varphi|$ a.e. yield

	\begin{align}\label{help}
	\int_{\Omega_t}\phi \langle N, D H\rangle dv& =\int_0^t\Phi(s)\Big(\int_{\Sigma_s}\nonumber
	\frac {\langle N, D H\rangle}{|D \varphi|} d\sigma_s\Big)ds\\ \nonumber
& 	+\int_0^t\Phi(s)\Big(\int_{\partial\Sigma_s}\frac {\langle N, D H\rangle }{|D \varphi|^2}D_\mu\varphi\; dl_s\Big)ds\\ 
	& =-\int_0^t\Phi(s)\int_{\partial\Si_s}\frac{D_\mu H}{H^2}dl_sds -\int_{\Omega_t}\phi |A|^2dv.
		\end{align}

On the other hand, 
	$$\mbox{div }(\phi H N)=(\Phi'\circ \varphi)H^2 +\phi \langle N, D H\rangle+\phi H^2.$$
	 So, integrating over $\Omega_t$ and combining (\ref{help}) give
	
	$$-\int_{\Omega_t}\langle D\phi, N\rangle H dv \leq\int_{\Omega_t}\phi(H^2-|A|^2)dv-\int_{\partial_{\partial C}\Omega_t}\phi\frac{D_\mu H}{H}dl_t.$$
	
	Denoting by $(\kappa_1,...,\kappa_{n-1})$ the principal
	curvature vector of $\Si_t,$ we have by the Newton-MacLaurin's inequality for the expression  $2\sum_{i<j}\kappa_i\kappa_j $ that 
	
	$$
	H^2-|A|^2\leq\frac{n-2}{n-1}H^2,
	$$
	with the equality holding only if $\Si_t$ is umbilical. 
	
	In view of the above, 
	
	$$
	-\int_{\Omega_t}\langle D\phi, N\rangle H dv\leq \frac{n-2}{n-1}\int_{\Omega_t}\phi H^2dv-\int_{\partial_{\partial C}\Omega_t}\frac{\phi}{H}D_\mu Hdl_t.
	$$
\end{proof}

We define the following quantity 
\begin{equation}\label{quant}
\mathcal{I}(\Si)=\int_{\Sigma}H d\sigma.
\end{equation}
Consider the foliation $\{\Sigma_t\}_{t\geq 0}$ defined by the level sets of the function given by weak solution of IMCF  for hypersurfaces with boundary in $C\backslash\Omega$. We prove the following proposition about the  functional $\mathcal{I}$.

\begin{proposition} \label{imcf} We have
	\begin{equation}\label{evol1}
	\mathcal{I}(\Si_t)\leq \mathcal{I}(\Si)\exp\Big[\Big(\frac{n-2}{n-1}\Big)\cdot t\Big],\;\; t\geq 0.
	\end{equation}	 \end{proposition}

\begin{proof}
	Using equality (3.3) of \cite{M}, we compute:
	\begin{equation*}
	\frac{d}{dt}\mathcal{I}(\Si_t)=\int_{\Si_t}\Big(H-\frac{|A|^2}{H}\Big)d\sigma_t+ \int_{\partial\Si_t}\frac{1}{H^2}D_{\mu}Hdl_t.
	\end{equation*}

Using once more the Newton-MacLaurin's inequality and the Neumann condition $D_{\mu}H=-H\Pi(N,N)$ derived in (3.6) of \cite{M3}, we get
	\begin{equation}\label{inqi}
	\frac{d}{dt}\mathcal{I}(\Si_t)\leq \frac{n-2}{n-1}\mathcal{I}(\Si_t)-\int_{\partial\Si_t}\frac{1}{H}\Pi(N,N)dl_t.
	\end{equation}
	
It remains to justify  (\ref{inqi}) even in the presence of jumps. In other words, when the flowing hypersurface
jumps, the total mean curvature strictly decreases 
and one can to extend  inequality (\ref{inqi}) through countably many jump times. 

Next consider the hypersurface  $\Sigma'=\partial_C\{\varphi>0 \}.$ From \cite[$\mathsection$4]{M3}, we know that $\Si'$ is strictly outward minimizing among hypersurfaces homologous to $\Si$. So define the following cutoff function 
$\Phi:[0,t]\to\mathbb R_+^n$ by 
\begin{equation*}
	\Phi(s)=\begin{cases}
	0\,&\text{on}\:\;[0,\tau]\\
	(s-\tau)/\varepsilon\,&\text{on}\:\;[\tau,\tau+\varepsilon]\\
	1&\text{on}\:\;[\tau+\varepsilon,t-\varepsilon]\\
	(t-s)/\varepsilon\,&\text{on}\:\;[t-\varepsilon,t].
	\end{cases}	
	\end{equation*}
	where  $\tau\in(0,t)$ and $\varepsilon\in(0,\frac{t-\tau}{2}).$ From Lemma \ref{tecn}, we have that $\phi=\Phi\circ \varphi$ satisfies 
	\begin{equation}\label{key}
	-\int_{\Omega_t}\langle D \phi, N\rangle Hdv 
	\leq \frac{n-2}{n-1}\int_{\Omega_t}\phi H^2dv-\int_{\partial_{\partial C}\Omega_t}\frac{\phi}{H}D_\mu Hdl.
	\end{equation}
	
	On the other hand,  making use of $D_\mu  H=-H\Pi(N,N)\leq0,$ $H = |D\varphi|$ a.e. and  the co-area formula we obtain the following inequality 
		\begin{align*}
	\frac{1}{\varepsilon}\Big[\int_{t-\varepsilon}^{\varepsilon}\mathcal{I}(\Si_s)ds-\int_{\tau}^{\tau+\varepsilon}\mathcal{I}(\Si_s)ds\Big]&=-\int_0^t\Phi'(s)\mathcal{I}(\Si_s)ds\\ 
	&= 
		-\int_{\Omega_t}\langle D \phi, N\rangle Hdv\\	
	&=
 \frac{n-2}{n-1}\int_{\Omega_t}\phi H^2dv \leq \frac{n-2}{n-1}\int_0^t\Phi(s)\mathcal{I}(\Si_s)ds.
	\end{align*}
	
	Taking the limit as $\varepsilon\to 0,$ we finally obtain
	$$\mathcal{I}(\Si_t)\leq\mathcal{I}(\Si_{\tau})
+ \frac{n-2}{n-1}\int_{\tau}^t\mathcal{I}(\Si_s)ds,\,\mbox{ for a.e. }0<\tau<t.$$
	
	For $n\leq7,$ Lemma $5.1$ and Lemma $5.5$ of \cite{M3} imply that 
	$$
	\Sigma_{\tau_i}\rightarrow \Sigma'\;\mbox{ as }\;\tau_i\searrow 0,
	$$
	locally in  $C^{1,\beta}$, $\beta<\frac{1}{2}.$ However, if $n>7$ and $\Sigma$ is free boundary outer-minimizing, then $\Sigma$ coincides with $\Si'$ and is disjoint from the singular set of  $\varphi,$ so the above convergence remains true.
	It follows from the  first variation formula of area  that $\Si_t$ possesses a weak mean curvature in $L^1$ (see (5.2) of \cite{M3}), which implies together the Riesz Representation Theorem that
	$$\mathcal{I}(\Si_{\tau_i})\to\mathcal{I}(\Si').$$
By the regularity result obtained in \cite{SWZ} and \cite{GJ} (see also Theorem 1.3 (iii) in \cite{HI}) and %the following fact about weak mean curvature
(see  (1.15) in \cite{HI})
 $$H_{\Sigma'}=0\; \mbox{ on } \;\Sigma'\backslash \Sigma\; \mbox{ and } \; H_{\Sigma'}=H_{\Sigma}\geq0\;\;\mbox{a.e.}\; \mbox{ on } \;\Sigma'\cap \Sigma,$$ 
 we conclude that 
	$\mathcal{I}(\Si')\leq \mathcal{I}(\Si).$ In particular,  
	$$\mathcal{I}(\Si_t)\leq\mathcal{I}(\Si)
	+ \frac{n-2}{n-1}\int_0^t\mathcal{I}(\Si_s)ds.\;\;\mbox{for a.e. $t>0$}.$$
	
Therefore the proposition follows directly  from Gronwall's Lemma.
	
\end{proof}

%%%%%%%%%%%%%%%%%%%%%%%%%%%%%%%%%%%%%%%%%%%%%
%%%%%%%%%%%%%%%%%%%%%%%%%%%%%%%%%%%%%%%%%%%%%

\section{Capacity inequalities }\label{icmtc}

%%%%%%%%%%%%%%%%%%%%%%%%%%%%%%%%%%%%%%%%%%%%%
%%%%%%%%%%%%%%%%%%%%%%%%%%%%%%%%%%%%%%%%%%%%%

Our first result is a capacity inequality for certain hypersurfaces with boundary in convex cones. In the sequel, we proof Theorem \ref{main3},  \ref{main4} and \ref{PFS}. 

First we recall that the infimum of (\ref{defcap}) is attained by a unique solution $\phi$ \footnote{The function $\phi$ is sometimes called of {\it electrostatic potential} of $\Si.$} of the following mixed boundary value 
problem in $C\backslash\overline\Omega$: 
\begin{equation}\label{mixed}
\begin{cases}
\Delta \phi =0\quad\text{in}\:\;C\backslash\overline\Omega\\
D_\mu \phi =0\quad\text{on}\:\;\partial_{\partial C}(C\backslash  \Omega)\\
\phi=0\;\;\text{on}\:\;\partial_{C}\Omega\;\;\text{and}\;\phi\to1\;\text{as}\;\;|x|\to\infty,
\end{cases}
\end{equation}
where $\mu$ is the outward unit normal to $\partial C$.  For a detailed proof of the existence and regularity of $\phi$, see \cite{IPS}\footnote{ In \cite{IPS}, the authors also provided a link between  Neumann parabolicity and  capacity of compact subsets.}.

\begin{proposition}\label{fenchel}
		Let $C\subset \mathbb{R}^{n}$ be a smooth  convex cone centered at the origin. Assume that  $\Omega$ is a smooth bounded domain containing the vertex of $C$ such that $\Sigma=\partial_C\Omega$ meets $\partial C$ orthogonally and is strictly mean convex. If 
		\begin{itemize}
			\item $n\leq7$; or
			\item $\Sigma$ is free boundary outer-minimizing in $C\backslash\Omega$.
		\end{itemize}
	 Then
	$$\mbox{Cap}(\Sigma)\leq \frac{1}{(n-1)\alpha_{n-1}}\mathcal{I}(\Si).$$
\end{proposition}
\begin{proof}
	We follow the ideas of P\'olya and Szeg\"o (\cite[$\mathsection$ 2]{PS}, see also \cite{BM}). Let $\varphi$ be a solution to the inverse mean curvature flow for hypersurfaces with boundary in $C\backslash\Omega$ with initial solution $\Si$. We set $\phi(x)=f\circ\varphi(x)$, where 
	$$
	f(t)=\Lambda\int_0^t\frac{1}{\mbox T(s)}ds,\quad\Lambda=\Big(\int_0^{\infty}\mbox T(t)^{-1}dt\Big)^{-1}.
	$$
	
	Using the co-area formula we  obtain 
	\begin{align}\label{vari}
	(n-2)\alpha_{n-1}\mbox{Cap}(\Sigma)&\leq \int_{C\backslash \Omega}|D \phi|^2dv\nonumber\\
	&= \int_0^{\infty}f'(t)^2 \mbox T(t)dt+\int_0^{\infty}f'(t) \int_{\partial\Si_t}D_\mu\varphi dl_t dt,
	\end{align}
	where $\mbox T(t)=\int_{\Sigma_t}|D\varphi| d\sigma.$
	
	From Proposition \ref{imcf} and the fact that $D_\mu\varphi=0$ on $\partial_{\partial C}(C\backslash  \Omega),$ we have   
	$$\left(\int_0^{\infty}\frac{1}{\mbox T(t)}dt\right)^{-1}\leq \frac{n-2}{n-1}\int_{\Sigma}Hd\sigma.$$
	Therefore
	$$\mbox{Cap}(\Sigma)\leq \frac 1{(n-1)\alpha_{n-1}}\mathcal{I}(\Si).$$
\end{proof}

\subsection{Proofs of Theorem \ref{main3},  \ref{main4} and \ref{PFS}:}
\begin{proof}[Proof of Theorem \ref{main3}] 
	By Proposition \ref{fenchel},
	\begin{equation}\label{torigid}
	\mbox{Cap}(\Sigma)\leq \frac{1}{(n-1)\alpha_{n-1}}\mathcal{I}(\Si).
	\end{equation}

	 When the equality holds in (\ref{torigid}), (\ref{evol1}) is in fact an equality. Thus, 
	$$H^2=(n-1)|A|^2\mbox{ on }\Sigma_t,\mbox{ for a.e. }t\geq 0\quad \mbox{and}\quad\Pi(N,N)=0,$$
	So $\Si_t$ is a union of  pieces of totally umbilical spheres, for a.e. $t\geq0.$ Note that the evolution by IMCF for hypersurfaces with boundary does not have any jumps and the flow remains classical( otherwise (\ref{evol1}) would be strictly). Another consequence is that  $\Sigma$ has to be  connected. 
	
	In the following,  this part of the proof follows ideas of  Theorem 4.9 in \cite{RS}.
	Let $L(\vec e)$ be the linear subspace generated by a vector $\vec e\in\mathbb R^{n}.$  Given $x\in\Si,$ the normal line $x+L(N(x))$ to $\Si$ contains the
	center $\mathbf{c}$ of a sphere containing $\Si.$ Take a point $\tilde x\in\partial\Si$ at maximum distance from the vertex of the
	cone. As $\Si$ meets $\partial C$ orthogonally,  we can assert that $N(\tilde x)$ is proportional to $\tilde x,$ and so $\mathbf{c}\in\partial C$.
	
	In order to prove that $\mathbf{c}=\vec{0}$, and thus $\Si$ is the intersection of a sphere centered at the vertex with the cone, we will argue by contradiction. Suppose that  $\mathbf{c}\neq\vec{0}$. Now, pick a point $x\in\partial \Si\backslash\vec 0$ and since $N(x)$ is tangent to $\partial C$ (by the orthogonality condition), we have that  $T_x\partial\Si$ contains the straight line $x+L(N(x))$. Hence  $\mathbf{c}$ and $x$ belong to $\partial C\cap T_x(\partial C)$  and, by convexity, the segment  line $l = \{ \mathbf{c}+tx;\; t\in[0,1]\}$ is contained in $\partial C.$ Therefore $\partial S\subset T_{\mathbf{c}}(\partial C)\cap \partial C$ and thus $\partial S$ is a great circle which bounds a flat region in $T_{x_0}(\partial C)\cap \partial C$, a contradiction.
	\end{proof}

\begin{proof}[Proof of Theorem \ref{main4}] 
	
	According to Marquadt \cite{M2}, consider the foliation $\{\Si_t\}_{t\geq0}$ given by IMCF in $C\backslash\Omega$ such that $\Si_0=\Si.$ 
	Define the function
	$$h(t):=\frac{\mathcal{I}(\Sigma_{t})}{\mbox{Area}(\Sigma_{t})^{\frac{n-2}{n-1}}}.$$
	
	We recall that the area element evolves in the normal direction by
	$$
	\frac{d}{dt}d\sigma_t=d\sigma_t.
	$$
	Therefore the area satisfies
	$$
	\frac{d}{dt}\mbox{Area}(\Si_t)=\mbox{Area}(\Si_t),\;\;\mbox{Area}(\Si_t)=\exp(t)\mbox{Area}(\Si_0),\;\;t\geq0,
	$$
	which together with Proposition \ref{imcf} implies that $h(t)$ is  non-increasing along IMCF in $C\backslash\Omega$. 
	From Theorem \ref{marqcone},  problem (\ref{icmf1})  has a unique smooth solution for all time and the rescaled hypersurface $\Sigma_t$  converges smoothly and exponentially  to a unique piece of the round sphere as $t\to \infty.$ 
Thus, at infinity $h(t)$ converges to $(n-1)\alpha_{n-1}^{1/(n-1)}.$

Therefore 
	$$\mbox{Area}(\Si)^{\frac{2-n}{n-1}}\mathcal{I}(\Sigma)=h(0)\leq h(t)\to(n-1)\alpha_{n-1}^{1/(n-1)}.$$
	From this, (\ref{al-fen})  follows easily.
	
	To prove the rigidity statement, we notice that if the inequality (\ref{al-fen}) becomes an equality we have $h(t)=h(0)$ for all $t$. Thus, (\ref{evol1}) is also an equality and, by reasoning as in proof of Theorem \ref{main3}, the rigidity follows.
\end{proof}

Inspired by the proof of Theorem 11 in \cite{J}, we prove

\begin{proof}[Proof of Theorem \ref{PFS}]
	Let $\varphi$ be a function satisfying (\ref{mixed}). Then
	\begin{align*}
	\mbox{Cap}(\Si) &=  \frac{1}{(n-2)\alpha_{n-1} }\int_{C \backslash\Omega} |D \varphi|^2 dv.
	\end{align*}
	Since the level sets of $\varphi$ give the foliation of $C\backslash\Omega$ we have by the co-area formula that
	\begin{equation*}
	\int_{C \backslash \Omega} |D \varphi|^2 dv = \int_0^1 \int_{\Sigma_t} |D \varphi| d\sigma_t dt+ \int_0^1 \Big(\int_{\partial\Sigma_t} D_\mu \varphi dl_t\Big) dt,
	\end{equation*}
	where $\Sigma_t=\varphi^{-1}(t)$ for $t \in [0,1)$.  
	
	We observe that on the right hand side we have
	\begin{equation}\label{sch}
	\int_{C \backslash \Omega} |D \varphi|^2 dv \geq \int_0^1 \frac{\mbox{Area}(\Sigma_t)^2}{\int_{\Sigma_t} |D \varphi|^{-1} d\sigma_t} dt.
	\end{equation}
	In addition, for our purposes it will be convenient to rewrite the integral on right hand side of (\ref{sch}) with help of the following expressions:
	\begin{itemize}
		\item  $\frac{d}{dt}\mbox{Vol}(\Omega_t)=\int_{\Sigma_t} \frac{1}{|D \varphi|} d\sigma_t+\int_{\partial\Sigma_t} \frac{1}{|D \varphi|^2 }D_\mu \varphi\; dL_t$
		\item $
		\mbox{Area}(\Sigma_t)\geq n\mbox{Vol}(C(\alpha,1))^{\frac{1}{n}}\mbox{Vol}(\Omega_t)^{\frac{(n-1)}{n}},
		$
	\end{itemize}
where the second line is the isoperimetric  inequality for convex cones proved by Lions and Pacella \cite{LP}).

	Thus we obtain that
	\begin{align*}
	\mbox{Cap}(\Si) \geq\frac{1}{(n-2)\alpha_{n-1} }\int_0^1 \frac{\alpha_{n-1}^2 \left(\frac{\mbox{Vol}(\Omega_t)}{\mbox{Vol}(C(\alpha,1))}\right)^{\frac{2(n-1)}{n}}}{d(\mbox{Vol}(\Omega_t))/dt} dt.
	\end{align*}
	
	We let $R(t)$ denote be the radius of the sphere whose intersection with the cone  has volume equal to $\mbox{Vol}(\Omega_t)=\mbox{Vol}(C(\alpha,1))R(t)^n$. Thus we get 
	\begin{equation}
	\label{capvol}
	\mbox{Cap}(\Si) \geq \frac{1}{(n-2)} 
	\int_0^1  \frac{R(t)^{n-1}}{R'(t)} dt.
	\end{equation}

	Let $\hat\Omega\subset C$  be an open sector with radius $R(t)$, vertex at origin and same volume as $\Omega$. Consider a function $\Psi:C\backslash\hat\Omega  \to\mathbb{R}^n$ such that $\Psi^{-1}(t)=\hat \Sigma_t$ and $D_\mu\Psi=0$ on $C\setminus\hat\Omega$, where $\hat  \Sigma_t=\partial_C\hat\Omega_t.$ 	Using that   $\mbox{Area}(\hat {\Sigma}_t)=\alpha_{n-1} R(t)^{n-1} $ and $|D \Psi| = \frac{1}{R'(t)}$ on $\hat \Sigma_t$, we can rewrite (\ref{capvol})  as
	
	\begin{align*}
	\mbox{Cap}(\Si) \geq \frac{1}{\alpha_{n-1}(n-2)} 
	\int_0^1  \int_{\hat  \Sigma_t}|D \Psi| d\hat\sigma_t dt.
	\end{align*}
 Finally, using once more the co-area formula we deduce that 
	$$\mbox{Cap}(\Si)\geq \mbox{Cap}(\hat \Sigma)=\left(\frac{\mbox{Vol}(\Omega)}{\mbox{Vol}(C(\alpha,1))}\right)^{\frac{n-2}{n}}.$$
	
The rigidity statement is consequence of the one given by the isoperimetric inequality.
\end{proof}

\begin{remark} Using alternative arguments, an another proof of the above theorem can be given. The idea is as follow. 
	Let $C$ be a convex cone. Assume that $\Omega\subset C$ is a smooth domain containing the vertex of $C$, $\hat\Omega$ is a convex sector with $\mbox{Area}(\hat\Omega)=\mbox{Area}(\Omega)$  and $u$ is  a function in $V^p(\Omega):=\{v\in H^{1,p}(\Omega),\;v=0\mbox{ on }\Sigma=\partial_C\Omega\}.$ One can define the  {\it Symmetrization} of $u$ as in \cite{ LP2,LP3} , which  is a transformation associating $u$  to a (unique) radial decreasing function $\hat u\in V^p(\hat\Omega)$ having the same distribution function as $u$. So the proof would follow along the same line as that given by Lemma $14$ and Lemma $15$ of \cite{S}. We also should mention that this symmetrization has the usual properties of the Schwarz symmetrization and can be used to prove classical isoperimetric inequalities for convex cones as in \cite{LP} or even to estimate the best Sobolev constant for embeddings \cite{LP2}. 
\end{remark}

%%%%%%%%%%%%%%%%%%%%%%%%%%%%%%%%%%%%%%%%%%%%
%%%%%%%%%%%%%%%%%%%%%%%%%%%%%%%%%%%%%%%%%%%%%

\section{Model case and mass-capacity inequalities}\label{Model}

%%%%%%%%%%%%%%%%%%%%%%%%%%%%%%%%%%%%%%%%%%%%%
%%%%%%%%%%%%%%%%%%%%%%%%%%%%%%%%%%%%%%%%%%%%%

We first give a brief discussion about asymptotic flatness and mass. A Riemannian manifold  $(M,g)$, $n\geq 3$, with a non-compact boundary $\partial M$  is {\em asymptotically flat}  if there exists a compact $K\subset M$ and a diffeomorphism $$\Psi: M\backslash K\to \mathbb R^n_+\cap \{|x|\leq 1\}$$  such that
$$
g_{ij}=\delta_{ij}+O(|x|^{-p}),
$$
and 
$$
\partial_kg_{ij}=O(|x|^{-p-1}) \quad\textrm{and}\quad \partial_l\partial_kg_{ij}=O(|x|^{-p-2})
$$
for some  $p>(n-2)/2$. The closed half-space $\mathbb R^n_+$ endowed with the standard flat metric $\delta$ is example of metric asymptotically flat. 

Assume that $R_g$ and $H_g$ are integrable in $M$ and $\partial M$, respectively.
In asymptotically flat coordinates, the mass of $(M,g)$ is given by
\begin{align}\label{massadm}
{\mathfrak m}(g)&:=&\lim_{r\to +\infty}\Big\{\frac{1}{2(n-1)\omega_{n-1}}\sum_{i,j=1}^{n}\int_{\{x\in \mathbb{R}^n_{+},\;|x|=r\}}(g_{ij,j}- g_{jj,i})\frac{x_i}{r} d\sigma_r \nonumber\\
& &+\frac{1}{2(n-1)\omega_{n-1}}\sum_{\beta=1}^{n-1}\int_{\{x\in \partial \mathbb{R}^n_{+},\;|x|=r\}} g_{\beta n}\frac{x_{\beta}}{r} d\sigma_r \Big\},
\end{align}
where $\omega_{n-1}$ is the volume of the $n-1$-dimensional sphere.

The above definition is independent of the particular choice of the chart at infinity what means that  the mass is a geometric invariant. As already mentioned, we also have a Positive mass theorem in this context (see \cite{ABdL}, Theorem 1.1), which states that an asymptotically flat manifold with non-negative scalar curvature and mean convex boundary has nonnegative mass if either $3\leq n\leq  7$ or $n\geq 3$ and $M$ is spin. Moreover the mass is zero if only if it is isometric to $\mathbb R^n_+$ with the flat metric.

In the following, we recall the definition of the
{\it half Schwarzschild space} of mass $\mathfrak m>0$ which is the set $\{x\in \mathbb R^n_+;|x|\geq (\mathfrak{m}/2)^{\frac{1}{n-2}}\}$ endowed with the following conformal metric 
$$
g_{\mathfrak{m}}=\left(1+\frac{\mathfrak m}{2}|x|^{2-n}\right)^{\frac{4}{n-2}}\delta,\;\;\mathfrak m>0. 
$$
This manifold is scalar-flat with a non-compact totally geodesic boundary $x_n=0$ and the coordinate hemisphere of radius $(\mathfrak{m}/2)^{\frac{1}{n-2}}$  is the unique free boundary area-minimizing horizon. A straightforward computation gives that the mass $\mathfrak m(g_m)$  is half the ADM mass of the standard Schwarzschild space. In fact, the double manifold of the half Schwarzschild space along its totally geodesic boundary is exactly the Schwarzschild space.

 We start by recalling some well known formulae. If  $g=u^{\frac4{n-2}}\delta$ for some positive smooth function $u$ on $M$, we know that 
\begin{equation}\label{conformal}
\begin{cases}
R_g=u^{-\frac{n+2}{n-2}}\Big(-\frac{4(n-1)}{(n-2)}\Delta u+R u\Big)\,\:\text{in}\: \; M \,\\
H_g=u^{-\frac{n}{n-2}}\Big( \frac{2(n-1)}{(n-2)}D_{\mu}u +H u\Big) \;\;\text{on}\: \; \partial M \,,
\end{cases}
\end{equation}
where $\mu$ and $\Delta$ denote the outward unit normal
vector to $\partial M$ and the Laplace-Beltrami operator with respect to $\delta$, respectively. 

As consequence of (\ref{conformal}), those geometric assumptions on metrics in $\mathcal M$  (i.e.,  scalar curvature $R_g\geq0$ and mean curvature $H_g\geq0$ ) are equivalents to assume  that  $\Delta u \leq0$ in $M$  and  $D_{\mu}u\geq0$ on $S$.

Another remarkable fact is  that  the maximum principle implies that a superharmonic
function $u$ on the half-space $\mathbb{R}_+^n$ which is $1$ at infinity and satisfies  $D_{x_n}u=0$ on $\partial \mathbb{R}_+^n$  is, in fact,  identically $1$. We emphasize that geometrically this means  we can not conformally deform  \footnote{In fact, one can not have any compact deformation which is a consequence of the positive mass theorem} the half-space standard metric in a bounded region  without
decreasing the scalar curvature or  mean curvature on the boundary somewhere.

Next we calculate the capacity of the Riemannian half Schwarzschild manifold. For $R>0$, consider a  function $u$ on $\{x\in \mathbb R^n_+;\;|x|\geq R=(\mathfrak{m}/2)^{\frac{1}{n-2}}\} $ 
defined by
 $$u=1+\left(R/|x|\right)^{n-2}.$$
 We notice that $u$ may be defined  and  harmonic in $\mathbb{R}_+^n\backslash\{0\}$  and satisfying $D_\mu u=0$ on $\partial\mathbb{R}_+^n\backslash\{0\}$  as well.

 Setting 
$$
\varphi=\frac{2-u}{u},
$$ 
we obtain by conformal change that $\Delta_g\varphi=0$ and   
$D_{\mu_{g}} \varphi=0$ on $S$. Note also that $\varphi=0$ on $\Sigma$ and $\varphi\to \infty$ at infinity. Thus, a straightforward calculation  implies that 
$\mbox{Cap}(\Sigma,g_\mathfrak{m})=\mathfrak{m}.$

Assume  that $(M,g)$ is asymptotically flat satisfying $R_g\equiv 0\equiv H_g$ near infinity, where $g=u^{\frac{4}{n-2}}\delta$. Thus,
\begin{equation*}
\begin{cases}
\Delta u =0& \;\;\text{in}\;\;\ \mathbb{R}^n_+ \\
D_{x_n} u=0  &\;\;\text{on}\;\;\partial \mathbb{R}^n_+\,,
\end{cases}
\end{equation*}
for $|x|$ large which allows to write 
\begin{equation}\label{exp:u}
u(x)=1+\frac{\mathfrak{m}}{2}|x|^{2-n}+O(|x|^{1-n}).
\end{equation}
 Such expansion might simplifies some calculations. So the next step is to state an important approximation lemma by harmonically flat metric at infinity.

\begin{lemma}(\mbox{Almaraz-Barbosa-de Lima} \cite{ABdL})\label{approxi}
	Let $(M,g)$ be an asymptotically flat manifold, conformally 
flat manifold with nonnegative scalar curvature $R_g\geq 0$ and  mean convex boundary $H_g\geq 0$. For  any $\epsilon>0$ small enough there exists an asymptotically flat $\bar{g}$ satisfying:
	\begin{itemize}
		\item[i)] $R_{\bar{g}}\geq 0$ and  $H_{\bar{g}}\geq 0$, with $R_{\bar{g}}\equiv 0$ and  $H_{\bar{g}}\equiv 0$  near infinity;
		\vspace{0,1cm}
		\item[ii)]  $\bar{g}$ is conformally flat near infinity;
\vspace{0,1cm}
		\item[iii)] $|\mathfrak m(\bar g)- \mathfrak m(g)|\leq \varepsilon$. 
	\end{itemize}
	
\end{lemma}

Now we can relate  the capacity defined by $g$ and the euclidean capacity. 

\begin{proposition} \label{prop} Let $g\in\mathcal{M}.$ If  $(M,g),\; n\ge 3,$ is asymptotically flat. Then  
	\begin{equation}\label{inqcap}
			\mbox{Cap}(\Sigma,g)\leq \mbox{Cap}(\Sigma)+\frac{\mathfrak{m}(g)}{2}.
	\end{equation}
	Equality holds if and only if $\Delta  u=0$ in $M$ and $D_\mu u=0$ on $S$.\\	
	
\end{proposition}

\begin{proof}		
	Assume that near infinity $u$ is harmonic and satisfies $D_\mu u=0$ on $S$. 
	Let us consider $w:\mathbb{R}^n_{+}\backslash \Omega \rightarrow (0,1)$ satisfying the mixed boundary value problem (\ref{mixed}) by replacing $C$ by $\mathbb{R}^n_{+}$. Now,  taking $\phi=\frac{w}{u}$ we have
	$$
	\int_{M}|D_g\phi|_g^2dv_g =\int_{\mathbb{R}^n_{+}\backslash \Omega}u^2\Big|D \Big( \frac{w}{u}\Big)\Big|^2dv   =\lim_{\rho\rightarrow \infty}\int_{I(\rho)}u^2\Big|D \Big( \frac{w}{u}\Big)\Big|^2dv,
	$$
	where $I(\rho):=\{x\in  \mathbb{R}_+^{n};\;|x|\leq \rho\}\backslash \Omega .$
	 %with. 
	Note that  
	\begin{align*}
		\int_{I(\rho) }u^2\Big|D \Big( \frac{w}{u}\Big)\Big|^2dv  &=\int_{I(\rho) } |D  w|^2-\Big\langle D(w^2),\frac{D u}u\Big\rangle dv  \\
		& +\int_{I(\rho) } w^2\frac{|D u|^2}{u^2}dv .
	\end{align*}
	
	Assume that $\rho$ is sufficiently large. The divergence theorem and the fact that $w=0$ on $\Si$ yield 
	\begin{align*}
	\int_{I(\rho)}\Big\langle D (w^2), \frac{D  u}u\Big\rangle dv  & =-\int_{I(\rho) } w^2\mbox{div}\Big(\frac{D  u}{u}\Big)dv  \\\nonumber
	&+ \int_{\mathbb{S}_+(\rho)}w^2\frac{D_r u }ud\sigma_{\rho}+\int_{\partial_{\partial \mathbb{R}_+^{n}}I(\rho)}w^2\frac{D_\mu u }ud\sigma_{\rho},\nonumber
	\end{align*} 
	where $r=|x|$ and $\mathbb{S}_+(\rho)$ is a large coordinate hemisphere of radius $\rho$. Since $\mbox{div}\Big(\frac{D  u}{u}\Big)=\frac{\Delta u}{u}-\frac{|D_{\mu}u|^2}{u},$ $\Delta u\leq 0$  in $\mathbb{R}^n_{+}\backslash \Omega$  and $D_\mu u\geq0$ on $\partial_{\partial\mathbb{R}^n_+}(\mathbb{R}^n_{+}\backslash \Omega)$ we conclude 
	
	\begin{align}\label{equ}
	\int_{I(\rho)}\Big\langle D (w^2), \frac{D  u}u\Big\rangle dv  \geq \int_{I(\rho) }w^2\frac{|D  u|^2}{u^2}dv +\int_{\mathbb{S}_+(\rho)}w^2\frac{D_r u }ud\sigma_{\rho}.
	\end{align}
	
	Combining all the facts above and taking the limit as $\rho\to\infty$, we get 
	$$2\int_{\mathbb{R}^n_{+}\backslash \Omega}u^2\Big|D \Big( \frac{w}{u}\Big)\Big|^2dv  \leq(n-2)\omega_{n-1}\Big(\mbox{Cap}(\Sigma)+\frac{\mathfrak{m}(g)}{2}\Big),$$
	and the inequality follows.\\
	
	Assume that (\ref{inqcap}) is an equality, but either  $\Delta 
	u<0$ in $\mathbb{R}^n_{+}\backslash \Omega$ or $D_\mu u >0$ on $\partial_{\partial\mathbb{R}^n_+}(\mathbb{R}^n_{+}\backslash \Omega).$ Therefore, suppose without loss of generality  that there exists $x_0\in\partial\mathbb{R}^n_+$  so that  $D_\mu u\ge c> 0$ on $\partial_{\mathbb{R}^n_+}(
	%\mathbb{R}^n_+\cap
	\{
	|x-x_0|\leq b\})$ for some constants $b,c>0.$

	By Lemma \ref{approxi}, we can take a sequence of metrics $\{g_k\},$ $g_k=u_k\delta,$ where each $u_k$ approximates of $u$. So, for $\rho,k\ge |x_0|+b$,  we can rewrite \eqref{equ}  as
	\begin{align*}
		\int_{I(\rho)}\Big\langle D w^2, \frac{D  u_k}{u_k}\Big\rangle dv  &\geq \int_{I(\rho) }w^2\frac{|D  u_k|^2}{u_k^2}dv  +\int_{\mathbb{S}_+(\rho)}w^2\frac{D_r u_k  }{u_k}d\sigma_{\rho}\\
		&+\int_{I(\rho,x_0)}w^2\frac{D_\mu u_k }{u_k}d\sigma_{b} \\
		&\ge   \int_{I(\rho) }w^2\frac{|D  u_k|^2}{u_k^2}dv  +\int_{\mathbb{S}_+(\rho)}w^2\frac{D_r u_k  }{u_k}d\sigma_{\rho}+C_0
	\end{align*}
	where $I(\rho,x_0)=\partial_{\mathbb{R}^n_+}(\{|x-x_0|\leq b\})$ and  $C_0>0$ is a  constant depending on $w,u,c,b$ and $n$. 
	Taking limits as $\rho\to\infty$ and  $k\to\infty$  we get

	$$\mbox{Cap}(\Sigma,g)\le \mbox{Cap}(\Sigma)+\frac{\mathfrak{m}(g)}{2}-2\omega_{n-1}^{-1}C_0.$$ 
	This leads to a contradiction because $C_0$ is positive. 
	
	The reciprocal is immediate because $\phi=\frac{w}{u}$ achieves the infimum for $\mbox{Cap}(\Si,g).$
	
\end{proof}

We are ready to prove a result that gives lower bounds for the mass in terms of the total mean curvature.

\begin{proposition}\label{import}
Let $g\in\mathcal{M}.$ If  $(M,g),\; n\ge 3,$ is asymptotically flat. We have  
	\begin{equation}\label{inq}
	\mathfrak{m}(g)\geq \frac{2\underline u }{(n-1)\omega_{n-1}}\int_{\Sigma}Hd\sigma,
	\end{equation}
where  $\underline{u}$ denotes the infimum of $u$ on $\Sigma$. 
The equality holds if and only if $\Delta_g  u=0$ in $M$, $D_\mu u=0$ on $S$ and $u_{|_{\Si}}$ is constant with $\underline u\geq2$.
\end{proposition}

\begin{proof} 

Because Lemma \ref{approxi}, we may assume that near infinity we have that $R_g\equiv 0$ and $H_g\equiv 0$  which together with  (\ref{conformal})  imply  that  near infinity $u$ is harmonic and satisfies $D_\mu u=0$ on $S$. 
	From (\ref{conformal}) and divergence theorem, we see that  
	\begin{align*}
	\int_{I(\rho)}&\Delta  
	u dv= \int_{ \mathbb{R}_+^{n}\cap\{|x|=\rho\}}{D_r u }d\sigma_{\rho}-\int_{\Sigma}{D_{N} u}d\sigma-\int_{\partial_{\partial \mathbb{R}_+^{n}}I(\rho)}{D_{\mu} u}d\sigma\\
	&= -\mathfrak{m}(g)\omega_{n-1}\frac{n-2}{4}+O\Big(\frac{1}{\rho}\Big)+\frac{n-2}{2(n-1)}\Big(\int_{\Sigma}Hud\sigma+\int_{\partial_{\partial \mathbb{R}_+^{n}}I(\rho)}Hud\sigma\Big).
	\end{align*}
%	where 
	%$V={\partial_{\partial \mathbb{R}_+^{n}}I(\rho)
		%\mathbb{R}_+^{n}\cap\{|x|\leq \rho\})\backslash \Omega 
	%	}.$
	
	In the limit as $\rho\to\infty$ we see that
	\begin{align}\label{mass}
	\mathfrak{m}(g)&=&	\frac{1}{\omega_{n-1}}\int_{\mathbb R^{n}_+\backslash\Omega} R_g u^{-1}dv +\frac{2}{(n-1)\omega_{n-1}}
	\int_{\Sigma}Hud\sigma\\\nonumber
	& &+\frac{2}{(n-1)\omega_{n-1}}
	\int_{\partial_{\partial \mathbb{R}_+^{n}}(\mathbb R^{n}_+\backslash\Omega)}Hud\sigma.
	\end{align}
	 So the inequality (\ref{inq}) holds provided $R_g\geq0$  and $S$ is minimal.
	
	Suppose that the equality holds in (\ref{inq}) and  either  $\Delta_g  u<0$ in $M$ or $D_\mu u>0$ on $S$ somewhere.  Take again a sequence of metrics $\{g_k\},$ $g_k=u_k\delta,$ as in the proof of Proposition \ref{prop}. Thus (\ref{mass}) becomes
	\begin{align*}
		\mathfrak{m}(g)<\frac{2}{(n-1)\omega_{n-1}}
		\int_{\Sigma}Hu_kd\sigma,
	\end{align*}
	taking the limit $k\to\infty$, this contradicts the fact that we are assuming the equality on (\ref{inq}).

	It is easy to see that $u$ is equals to its minimum on $\Sigma$, but it remains to show that $\underline u\geq 2$. Indeed, an analogous calculation using the divergence theorem gives 
	\begin{align*}
		\int_{\mathbb R^n_+\backslash\Omega}u^2R_g dv&=-\int_{\mathbb R^n_+\backslash\Omega}|D u|^2dv- \omega_{n-1}(n-2)\frac{\mathfrak{m}(g)}{4}\\
		&+\underline u^2\frac{(n-2)}{2(n-1)} \Big(\int_{\Sigma}Hd\sigma+\int_{\partial_{\partial\mathbb R^n_+}(\mathbb R^n_+\backslash\Omega)}Hd\sigma\Big).
	\end{align*}
	
	Note that since  
	$$2\int_{\mathbb R^n_+\backslash\Omega}|D u|^2dv=(n-2)\omega_{n-1}(\underline u-1)^2\mbox{Cap}(\Sigma),$$
	 we can conclude using the equality in (\ref{inq}) that 
	\begin{equation}
	(\underline u -1)\mbox{Cap}(\Sigma)=\frac{\mathfrak{m}(g)}{2}\geq\frac{\underline u}{2}\mbox{Cap}(\Si).
	\end{equation}
	Therefore we have $\underline u\geq 2$. 
	
\end{proof}

Putting the above inequalities we can prove the main result of this section.

\begin{proof}[Proof of Theorem \ref{m}]  
	
	From Proposition \ref{fenchel}, \ref{prop} and \ref{import} it follows that 
	
	$$
	\mbox{Cap}(\Sigma,g)\le \mbox{Cap}(\Sigma)+\frac{\mathfrak{m}(g)}{2}\leq \mathfrak{m}(g).
	$$
	Therefore, we obtain  (\ref{ineq1}).
	
	The volumetric Penrose inequality (\ref{ineq2}) follows from  Theorem \ref{PFS},  Proposition \ref{fenchel} and \ref{import}:
		$$\frac{\mathfrak{m}(g)}{2}\geq\frac{2}{(n-1)\omega_{n-1}}\int_\Si Hd\sigma\geq\Big(\frac{\mbox{Vol}(\Omega)}{\mbox{Vol}(\mathbb{R}^n_+\cap \{|x|\leq1\})}\Big)^{\frac{n-2}{n}}.$$
	
	%We will now prove the rigidity statement. 
	To complete the proof, we must consider the cases of equalities. Suppose that (\ref{ineq1}) becomes equality. In particular (\ref{inq}) is also an equality and thus  we can apply Proposition \ref{import} to get  
	\begin{equation}\label{mixedsch}
	\begin{cases}
	\Delta_g  u=0\quad\text{in}\:\;M\\
	D_\mu u=0\quad\text{on}\:\;S\\
	u\equiv2\;\;\text{on}\:\;\Si.
	\end{cases}
	\end{equation}
	
    According to Theorem \ref{main3}, $\Si$ is a hemisphere. This together with (\ref{mixedsch}) imply that $(M,g)$ is isometric to the Riemannian half Schwarzschild manifold.
	 
	 Arguing similarly, if  the equality occurs in (\ref{ineq2}) we also have that $g$ is isometric to the Riemannian half Schwarzschild metric.	
	 
\end{proof}

\noindent{\bf Ackwnoledgement}\\
  I would like to thank  Professor A. Neves  for providing a wonderful scientific environment when I was visiting Imperial College London and where the first drafts of this work were written. Also, I would like to thank L. Pessoa for bringing \cite{IPS} to my  attention. %and for  discussions about capacity. 
While at Imperial College, I was supported by  CNPq/Brazil

\bibliographystyle{amsbook}

\end{document}